\documentclass{amsart}
\usepackage{mathrsfs,amssymb,mathrsfs, multirow,setspace}
\usepackage[a4paper, total={7in, 9in}]{geometry}
\usepackage{enumerate}
\theoremstyle{plain}
\usepackage{graphicx}
\newtheorem{theorem}{Theorem}[section]
\newtheorem{lemma}[theorem]{Lemma}
\newtheorem{proposition}[theorem]{Proposition}
\newtheorem{corollary}[theorem]{Corollary}

\theoremstyle{definition}

\theoremstyle{remark}
\newtheorem{remark}[theorem]{Remark}

\input xy
\xyoption{all}

\def\G{\Gamma}

\begin{document}
	%\title[Equality of various graphs on finite semigroup]{Equality of various graphs on finite semigroup}
	%\maketitle
	%\section*{Introduction}
	%******************
	\title[ On the intersection ideal graph of semigroups]{On the intersection ideal graph of semigroups}
	\author[Barkha Baloda, Jitender Kumar]{Barkha Baloda, $\text{Jitender Kumar}^{^*}$}
	\address{Department of Mathematics, Birla Institute of Technology and Science Pilani, Pilani, India}
	\email{barkha0026@gmail.com,jitenderarora09@gmail.com}

	%\date{...}
	\begin{abstract}
		The intersection ideal graph $\Gamma(S)$ of a semigroup $S$ is a simple undirected graph whose vertices are all nontrivial left ideals of $S$ and two distinct left ideals $I, J$ are adjacent if and only if their intersection is nontrivial. In this paper, we investigate the connectedness of $\Gamma(S)$. We show that if $\Gamma(S)$ is connected then $diam(\Gamma(S)) \leq 2$. Further we classify the semigroups such that the diameter of their intersection graph is two. Other graph invariants, namely perfectness, planarity, girth, dominance number, clique number, independence number etc. are also discussed. Finally, if $S$ is union of $n$ minimal left ideals then we obtain the automorphism group of $\Gamma(S)$.
		\end{abstract}

	
	\keywords{Semigroup, ideals, clique number, graph automorphism\\ *  Corresponding author}
	
	\maketitle
	\section{Introduction}
Literature is abound with numerous remarkable results concerning a number of constructions of  graphs from rings, semigroups or groups. %\cite{abdollahi2006non, i.Bosak, cameron2011power,  demeyer2002zerodivisor, transitivekelarev, combinatorialkelrav, a.Cayley-abundant, redmond2003ideal}. The investigation of graphs related to various algebraic structures is very important because graphs of this type have valuable applications and are related to automata theory(see \cite{b.kelarev2003graph, a.kelarev2009cayley, kelarevminimalautomata}).
%Akbari et al. %\cite{akbari2014inclusion} 
%introduced the inclusion ideal graph associated with ring structures. 
The  intersection graph of a  semigroup was introduced by Bos\'ak  \cite{i.Bosak} in $1964$.
 The \emph{intersection subsemigroup graph} $\Gamma(S)$  of $S$ is an undirected simple graph whose vertex set is the collection of proper subsemigroups of  $S$ and two distinct vertices $A, B$ are adjacent if and only if $A \cap B \neq \emptyset$. %\cite{ inclusionidealringakbari} 
In \cite{i.Bosak}, it was shown that if $S$ is a nondenumerable semigroup or a periodic semigroup with more than two elements, then the graph $\Gamma(S)$ is connected. Bos\'ak then  raised the following  open problem: Does there exists a semigroup with more than two elements whose graph is disconnected? Y. F. Lin \cite{lin1969}, answer the problem posed by Bos\'ak, in the negative manner and proved that every semigroup with more than two elements has a connected graph. Also, B. Pond\v{e}li\v{c}ek \cite{abc} proved that  the diameter of a semigroup with more than two  elements does not exceed three.

Inspired by the work of J. Bos\'ak , Cs\'ak\'any and Poll\'ak \cite{a.Cskany1969} studied the intersection graphs of groups and showed that there is an edge between two proper subgroups if they have at least two elements common. Further, Zelinka \cite{B.kaka}, continued the work for finite abelian groups. R. Shen \cite{R.shen}, characterized all finite groups whose intersection graphs are disconnected. This solves the problem posed in \cite{a.Cskany1969}. %S. H. Jafari \emph{et al.} classified 
The groups whose intersection graphs of normal subgroups are connected, complete, forests or bipartite  are classified in \cite{a.jafari}. Tamizh \emph{et al.} \cite{T.Chelvam2012},  continued the seminal paper of Cs\'ak\'any and Poll\'ak to introduce the subgroup intersection graph of a finite group $G$. Further, in \cite{X.Ma}, it was shown that the diameter of intersection graph of a finite non-abelian simple group has an upper bound $28$. Shahsavari \emph{et al.} \cite{a.Shahsavari2017} have studied the structure  of the automorphism group of this graph. The intersection graph on cyclic subgroups of a group has been studied in \cite{a.Haghi2017}. Further, Kayacan \emph{et al.} \cite{kayacan2015abelian} studied the conjecture given in \cite{B.kaka}, that two (noncyclic) finite abelian groups with isomorphic intersection graphs are isomorphic. In \cite{kayacan2018connectivity}, finite solvable groups whose intersection graphs are not 2-connected, finite nilpotent groups whose intersection graphs are not 3-connected is classified. Further, the dominating sets of the intersection graph of finite groups is investigated in \cite{a.kalyacan}.

Recently, Chakrabarty et al. \cite{a.sen2009}  introduced the notion of  intersection ideal graph of rings. The \emph{intersection ideal graph} $\Gamma(R)$  of a ring $R$ is an undirected simple graph whose vertex set is the collection of nontrivial left ideals of  $R$ and two distinct vertices $I, J$ are adjacent if and only if $I \cap J \neq \{0\}$. They characterized the rings $R$ for which the graph $\Gamma(R)$ is connected and obtain several necessary and sufficient conditions on a ring $R$ such that $\Gamma(R)$ is complete. Planarity of intersection graphs of ideals of ring with unity is described in \cite{MR2660547} and domination number in \cite{jafari2011dominion}. Akbari \emph{et al.} \cite{S.Akbari2013} classified all rings whose intersection graphs of ideals are not connected and also determined all rings whose clique number is finite. The intersection graphs of ideals of direct product of rings have been discussed in \cite{MR3310566}. Pucanovic \emph{et al.} \cite{MR3190084} classified all graphs of genus two that are intersection graphs of ideals of some commutative rings and obtain some lower bounds for the genus of the intersection graph of ideals of a non local commutative ring. In \cite{das2017}, Das  characterized the positive integer $n$ for which the intersection graph of ideals of $\mathbb{Z}_n$ is perfect. The Intersection graph for submodules of modules have been studied in \cite{akbari2012intersection,akbari2017some, yaraneri2013intersection}. The intersecton graph on algebraic structures have also been studied in \cite{ahmadi2016planarity,  akbari2015intersectiongroup, akbari2014some, jafari2011results,  laison2010subspace,  xu2020automorphism}.

It is pertinent as well as interesting to associate graphs to ideals of a semigroup as ideals gives a lot of information about the structure of semigroups. %\cite{saito1958semigroups, satyanarayana1978structure}. 
Motivated with the work of \cite{S.Akbari2013, a.sen2009},
in this paper, we consider the intersection ideal graph associated with semigroups. The \emph{intersection ideal graph} $\Gamma(S)$ of a semigroup $S$ is an undirected simple graph whose vertex set is nontrivial left ideals of  $S$ and two distinct nontrivial left ideals $I, J$ are adjacent if and only if their intersection is nontrivial. The paper is arranged as follows. In Section 2, we state necessary fundamental notions  and recall some necessary results. Section 3 comprises the results concerning the connectedness of intersection ideal graph of an arbitrary semigroup. In Section 4, we study various graph invariants of  $\Gamma(S)$ viz. girth, dominance number, independence number and clique number etc.  Further, if $S$ is union of $n$ minimal left ideals then the automorphism group of $\Gamma(S)$ is obtained.
	%%%%%%%%%%%%%%%%%%%%%%%%%%%%%%%%%%%%%%%%%%%%%%%%%%%%%%%%%%%%%%%%%%%%%%%%%%%%%%%%%%%%%%%%%%%%%%%%%%%%%%%%%%%%%%%%%%%%%%%%%%%%%%%%%%%%%%%%%%%%%%%%%%%%%%%%%%%%%%%%%%%%%%%%%%%%%%%%%
	\section{Preliminaries}
In this section, first we recall necessary definitions and results of semigroup theory from \cite{b.clifford61vol1}. A \emph{semigroup} $S$ is a non-empty set together with an associative binary operation on $S$. The Green's $\mathcal{L}$-relation on a semigroup $S$ defined as $x$         $\mathcal{L}$ $y \Longleftrightarrow S^{1}x = S^{1}y$ where $S^{1}x = Sx \cup    \{x\}$. %most fundamentals tools in understanding a semigroup are its Green's relations: $\mathcal{L}$, $\mathcal{R}$, $\mathcal{D}$, $\mathcal{J}$ and $\mathcal{H}$ such that $x \mathcal{L} y \Longleftrightarrow S^{1}x = S^{1}y$, $x \mathcal{R} y \Leftrightarrow xS^{1} = yS^{1}$ and $x \mathcal{J} y \Leftrightarrow S^{1}xS^{1} = S^{1}yS^{1}$. 
	 The $\mathcal{L}$-class of an element $a \in S$ is denoted by $L_a$. 	A non-empty subset $I$ of $S$ is said to be a \emph{left [right] ideal} if $SI \subseteq I  [IS \subseteq I]$ and an  \emph{ideal} of $S$ if $SIS \subseteq I$. Union of two left [right] ideals of $S$ is again a left [right] ideal of $S$. A left ideal $I$ is \emph{maximal} if it does not contained in  any nontrivial left ideal of $S$. If $S$ has a unique maximal left ideal then it contains every nontrivial left ideal of $S$.
	A left ideal $I$ of $S$ is \emph{minimal} if it does not properly contain any left ideal of $S$. It is well known that every non-zero element of a minimal left ideal of $S$ is in same $\mathcal{L}$-class. If $S$ has a minimal left ideal then every nontrivial left ideal contains at least one minimal left ideal. If $A$ is any other left ideal of $S$ other than $I$, then either $I \subset A$ or $I \cap A = \emptyset$. Thus we have the following remark.
	\begin{remark}\label{disjoint intersection minimal}
		Any two different minimal left ideals of a semigroup $S$ are disjoint.
	\end{remark} 
	\begin{remark}\label{everynontrivial left ideal is union}
	Let $S$ be union of $n$ minimal left ideals. Then each nontrivial left ideal is union of these minimal left ideals.
	\end{remark}
The following lemma is useful in the sequel and we shall use this without referring to it explicitly.
	\begin{lemma}\label{S minus K is lclass}
A left ideal $K$ of $S$ is maximal if and only if $S \setminus K$ is an $\mathcal{L}$-class.
\end{lemma}
\begin{proof}
First suppose that $S \setminus K$ is an $\mathcal{L}-$class. Let if possible, $K$ is not maximal left ideal of $S$. Then there exists a nontrivial left ideal $K'$ of $S$ such that $K \subset K'$. There exists $a \in K'$ but $a \notin K$. Thus,  $L_a = S \setminus K$. Consequently, $L_a \subset K'$ gives $S = K'$, a contradiction. %Thus $K$ is a maximal left ideal of $S$. 
Conversely, suppose that $K$ is a maximal left ideal of $S$. For each $a \in S \setminus K$, maximality of $K$ implies $K \cup S^1a = S$. Consequently, $a$  $\mathcal{L}$  $b$ for every $a, b \in S \setminus K$. Thus $S \setminus K$ is contained in some $\mathcal{L}-$class and this $\mathcal{L}-$class is disjoint from $K$. It follows that $S \setminus K$ is an $\mathcal{L}-$class.   
\end{proof}

	We also require the following graph theoretic  notions \cite{westgraph}. A \emph{graph} $\Gamma$ is a pair  $\Gamma = (V, E)$, where $V = V(\Gamma)$ and $E = E(\Gamma)$ are the set of vertices and edges of $\Gamma$, respectively. We say that two different vertices $u, v$ are $\mathit{adjacent}$, denoted by $u \sim v$ or $(u,v)$, if there is an edge between $u$ and $v$. We write $u \nsim v$, if there is no edge between $u$ and $v$. The \emph{distance} between two vertices $u, v$ in $\Gamma$ is the number of edges in a shortest path connecting them and it is denoted by $d(u, v)$. If there is no path between $u$ and $v$, we say that the distance between  $u$ and $v$ is \emph{infinity} and we write as $d(u, v) = \infty$. The diameter $diam(\Gamma)$ of $\Gamma$ is the greatest distance between any pair of vertices. The \emph{degree} of the vertex $v$ in $\Gamma$ is the number of edges incident to $v$ and it is denoted by $deg(v)$. %The minimum degree is denoted by \emph{$\delta (\Gamma)$}.
	%A graph $\Gamma$ is \emph{regular} if degree of every vertex is same. A graph $\Gamma$ is said to be \emph{biregular} if all vertices have two distinct degrees. 
	A \emph{cycle} is a closed walk with distinct vertices except for the initial and end vertex, which are equal and a cycle of length $n$ is denoted by $C_n$. The \emph{girth} of $\Gamma$ is the length of its shortest cycle and is denoted by ${g(\Gamma)}$. A subset $X$ of $V(\Gamma)$ is said to be \emph{independent} if no two vertices of $X$ are adjacent. The \emph{independence number} of $\Gamma$  is the cardinality of the largest independent set and it is denoted by $\alpha(\Gamma)$. A graph $\G$ is \emph{bipartite}  if $V(\Gamma)$ is the union of two disjoint independent  set. It is well known that a graph is bipartite if and only if it has no odd cycle {\cite[Theorem 1.2.18]{westgraph}}. A connected graph $\Gamma$ is Eulerian if and only if degree of every vertex is even {\cite[Theorem 1.2.26]{westgraph}}. A \emph{subgraph}  of $\Gamma$ is a graph $\Gamma'$ such that $V(\Gamma') \subseteq V(\Gamma)$ and $E(\Gamma') \subseteq E(\Gamma)$.  A subgraph $\Gamma'$ of  $\Gamma$ is called an \emph{induced subgraph} by the elements of  $V(\Gamma') \subseteq V(\Gamma)$ if for $u, v \in V(\Gamma')$, we have $u \sim v$   in $\Gamma'$ if and only if $u \sim v$ in $\Gamma$.  The \emph{chromatic number} of $\Gamma$, denoted by $\chi(\Gamma)$, is the smallest number of colors needed to color the vertices  of $\Gamma$ so that no two adjacent vertices share the same color. A \emph{clique} in $\Gamma$ is a set of pairwise adjacent vertices. The \emph{clique number} of $\Gamma$ is the size of maximum clique in $\Gamma$ and it is denoted by $\omega(\Gamma)$. It is well known that $\omega(\Gamma) \leq \chi(\Gamma)$ (see \cite{westgraph}). A graph $\Gamma$ is \emph{perfect} if $\omega(\Gamma') = \chi(\Gamma')$ for every induced subgraph $\Gamma'$ of $\Gamma$.
	Recall that the {\em complement} $\overline{\Gamma}$ of $\Gamma$ is a graph with same vertex set as $\Gamma$ and distinct vertices $u, v$ are adjacent in $\overline{\Gamma}$ if they are not adjacent in $\Gamma$. A subgraph $\Gamma'$ of $\Gamma$ is called \emph{hole} if $\Gamma'$  is a  cycle as an induced subgraph, and $\Gamma'$  is called an \emph{antihole} of   $\Gamma$ if   $\overline{\Gamma'}$ is a hole in $\overline{\Gamma}$.

	\begin{theorem}\label{strongperfecttheorem}\cite{strongperfectgraph}
		A finite graph $\Gamma$
		is perfect if and only if it does not contain hole or antihole of odd length at least $5$.
	\end{theorem}
	%%%%%%%%%%%%%%%%%%%%%%%%%%%%%%%%%%%%%%%%%%%%%%%%%%%%%%%%%%%%%%%%%%%%%%%%%%%%%%%%%%%%%%%%%%%%%%%%%%%%%%%%%%%%%%%%%%%%%%%%%%%%%%%%%%%%%%%%%%%%%%%%%%%%%%%%%%%%%%%%%%%%%%%%%%%%%%%%%
	A subset $D$ of $V(\Gamma)$ is said to be a dominating set if any vertex in $V(\Gamma) \setminus D$ is adjacent to at least one vertex in $D$. If $D$ contains only one vertex then that vertex is called dominating vertex. The \emph{domination number} $\gamma(\Gamma)$ of $\Gamma$ is the minimum size of a dominating set in $\Gamma$. A graph $\Gamma$ is said to be planar if it can be drawn on a plane without any crossing of its edges. In $\Gamma$, a vertex $z$ resolves a pair of distinct vertices $x$ and $y$ if
	$d(x, z) \neq d(y, z)$. A resolving set of $\Gamma$ is a subset $R \subseteq V (\Gamma)$ such that every pair of distinct vertices of $\Gamma$ is resolved by some vertex in $R$. The metric dimension of $\Gamma$,
	denoted by $\beta(\Gamma)$, is the minimum cardinality of a resolving set of $\Gamma$. For  vertices $u$ and $v$ in a graph $\Gamma$, we say that $z$ \emph{strongly resolves} $u$ and $v$ if there exists a shortest path from $z$ to $u$ containing $v$, or a shortest path from $z$ to $v$ containing $u$. A subset $U$ of $V(\Gamma)$ is a \emph{strong resolving set} of $\Gamma$ if every pair of vertices of $\Gamma$ is strongly resolved by some vertex of $U$. The least cardinality of a strong resolving set of $\Gamma$ is called the \emph{strong metric dimension} of $\Gamma$ and is denoted by $\operatorname{sdim}(\Gamma)$. For  vertices $u$ and $v$ in a graph $\Gamma$, we write $u\equiv v$ if $N[u] = N[v]$. Notice that that $\equiv$ is an equivalence relation on $V(\Gamma)$.
We denote by $\widehat{v}$ the $\equiv$-class containing a vertex $v$ of $\Gamma$.
 Consider a graph $\widehat{\Gamma}$ whose vertex set is the set of all $\equiv$-classes, and vertices $\widehat{u}$ and  $\widehat{v}$ are adjacent if $u$ and $v$ are adjacent in $\Gamma$. This graph is well-defined because in $\Gamma$, $w \sim v$ for all $w \in \widehat{u}$ if and only if $u \sim v$.  We observe that $\widehat{\Gamma}$ is isomorphic to the subgraph $\mathcal{R}_{\Gamma}$ of $\Gamma$ induced by a set of vertices consisting of exactly one element from each $\equiv$-class. Subsequently, we have the following result of \cite{ma2018strong} with $\omega(\mathcal{R}_{\Gamma})$ replaced by $\omega(\widehat{\Gamma})$.

\begin{theorem}[{\cite[Theorem 2.2]{ma2018strong}}]\label{strong-metric-dim}
For any graph $\Gamma$ with diameter $2$,  $\operatorname{sdim}(\Gamma) = |V(\Gamma)| - \omega(\widehat{\Gamma})$.
\end{theorem}%A graph $\Gamma$ is said to be \emph{triangulated} if any vertex is a vertex of a triangle. %A \emph{vertex (edge) cutset} in a connected graph $\Gamma$ is a set of vertices (edges) whose deletion increases the number of connected components of $\Gamma$. The \emph{vertex connectivity} (\emph{edge connectivity}) of a connected graph $\Gamma$ is the minimum size of a vertex (edge) cutset and it is denoted by $\kappa (\Gamma)$ $ ({\kappa'}(\Gamma))$. It is well known that $\kappa (\Gamma)$ $\leq$ ${\kappa'}(\Gamma)$ $\leq$ $\delta({\Gamma})$ (see {\cite[Theorem 4.1.9]{westgraph}}). 

%%%%%%%%%%%%%%%%%%%%%%%%%%%%%%%%%%%%%%%%%%%%%%%%%%%%%%%%%%%%%%%%%%%%%%%%%%%%%%%%%%%%%%%%%%%%%%%%%%%%%%%%%%%%%%%%%%%%%%%%%%%%%%%%%%%%%%%%%%%%%%%%%%%%%%%%%%%%%%%%%%%%%%%%%%%%%%%%%

	%%%%%%%%%%%%%%%%%%%%%%%%%%%%%%%%%%%%%%%%%%%%%%%%%%%%%%%%%%%%%%%%%%%%%%%%%%%%%%%%%%%%%%%%%%%%%%%%%%%%%%%%%%%%%%%%%%%%%%%%%%%%%%%%%%%%%%%%%%%%%%%%%%%%%%%%%%%%%%%%%%%%%%%%%%%%%%%%%
	\section{Connectivity of the Intersection graph $\Gamma(S)$}
	In this section, we investigate the connectedness of $\Gamma(S)$. We show that $diam(\Gamma(S)) \leq 2$ if it is connected. Also, we classify the semigroups, in terms of left ideals, such that the diameter of $\Gamma(S)$ is two.
	\begin{theorem}\label{disconnectedintersection}
		The intersection ideal graph $\Gamma(S)$ is disconnected if and only if $S$ contains at least two minimal left ideals and every nontrivial left ideal of $S$ is minimal as well as maximal.
	\end{theorem}
	\begin{proof}
	    First suppose that $\Gamma(S)$ is not connected. Then $S$ has at least two nontrivial left ideals, namely $I_1, I_2$. Without loss of generality, assume that $I_1 \in C_1$ and $I_2 \in C_2$, where $C_1$ and $C_2$ are distinct components of $\Gamma(S)$.  If $I_1$ is not minimal then there exists at least one nontrivial left ideal $I_k$ of $S$ such that $I_k \subset I_1$ so that their intersection is nontrivial. Therefore, $I_1 \sim I_k$. Now if the intersection of $I_2$ and $I_k$ is nontrivial then $I_1 \sim I_k \sim I_2$, a contradiction. Therefore the  intersection of $I_2$ and $I_k$ is trivial. If $I_2 \cup I_k \neq S$ then $I_1 \sim I_2 \cup I_k \sim I_2$, a contradiction. Thus, $I_k \cup I_2 = S$. It follows that $I_1 \sim I_2$, again a contradiction. Thus $I_1$ is minimal. Similarly, we get $I_2$ is minimal. 
	    
	    Further assume that $I_1$ is not maximal. Then  there exists a nontrivial left ideal $I_k$ of $S$ such that $I_1 \subset I_k$ so that $I_1 \sim I_k$. If $I_1 \cup I_2 \neq S$ then $I_1 \sim I_1 \cup I_2 \sim I_2$, a contradiction to the fact that $\Gamma(S)$ is disconnected. It follows that $I_1 \cup I_2 = S$ so that the intersection of $I_k$ and $I_2$ is nontrivial. Thus we have $I_1 \sim I_k \sim I_2$, a contradiction. Hence $I_1$ is maximal. Similarly, we observe that $I_2$ is maximal. The converse follows from the Remark \ref{disjoint intersection minimal}.
	\end{proof}
		\begin{corollary}\label{null graphintersection}
If the graph $\Gamma(S)$ is disconnected then it is a null graph (i.e. it has no edge). 
\end{corollary}

\begin{theorem}\label{two minimalintersection}
		The graph $\Gamma(S)$ is disconnected if and only if $S$ is the union of exactly two minimal left ideals.
	\end{theorem}
\begin{proof}
		Suppose first that $\Gamma(S)$ is disconnected. Then by Theorem \ref{disconnectedintersection}, each nontrivial left ideal of $S$ is minimal. Suppose $S$ has at least three minimal left ideals, namely $I_1, I_2$
		and $I_3$. Then $I_1 \cup I_2$ is a nontrivial left ideal of $S$ which is not minimal. Consequently, by Theorem \ref{disconnectedintersection}, we get a contradiction of the fact that $\Gamma(S)$ is disconnected. Thus, $S$ has exactly two minimal left ideals. If $S \neq I_1 \cup I_2$, then $I_1 \cup I_2$ is a nontrivial left ideal which is not minimal, a contradiction ( cf. Theorem \ref{disconnectedintersection}). Thus, $S = I_1 \cup I_2$.
		
		Conversely, suppose $S = I_1 \cup I_2$ where $I_1$ and  $I_2$ are minimal left ideals of $S$. If there exists another nontrivial left ideal $I_k$ of $S$ then either $I_1 \subset I_k$ or $I_2 \subset I_k$. Without loss of generality, assume that $I_1 \subset I_k$, we have $I_1 \sim I_k$. Since $I_1 \cup I_2 = S$ we get $I_k \cup I_2 = S$. It follows that the intersection of $I_2$ and $I_k$ is nontrivial. By minimality of $I_2$, we can observe that $I_2 \subset I_k$. Consequently, $S \subseteq I_k$, a contradiction.  Thus, by Theorem \ref{disconnectedintersection},  $\Gamma(S)$ is disconnected.
	
\end{proof}
	\begin{theorem}\label{diameter2}
		If $\Gamma(S)$ is a connected graph then $diam(\Gamma(S))$ $\leq$ $2$.	
	\end{theorem}
	\begin{proof}
		Let $I_1, I_2$ be two nontrivial left ideals of $S$. If $I_1 \sim I_2$ then $d(I_1, I_2)$ = 1. If $I_1 \nsim I_2$ i.e. $I_1 \cap I_2$ is trivial then in the following cases we show that $d(I_1, I_2) $$\leq 2$.
		
		\noindent\textbf{Case 1.} $I_1 \cup I_2 \neq S$. Then $I_1 \sim (I_1 \cup I_2) \sim I_2$ so that $d(I_1, I_2)$ = 2.
		
		\noindent\textbf{Case 2.} $I_1 \cup I_2 = S$. Since $\Gamma(S)$ is a connected graph, there exists a nontrivial left ideal $I_k$ of $S$ such that either $I_1 \cap I_k$ is nontrivial or $I_2 \cap I_k$ is nontrivial. Now we have the following subcases.
		
		\textbf{Subcase 1.} $I_1 \not \subset I_k$ and $I_k \not \subset I_1$. Since $I_1 \not \subset I_k$ it follows that there exists $x \in I_k$ but $x \notin I_1$ so that $x \in I_2$. Consequently, $I_2 \cap I_k$ is nontrivial. Therefore, we get a path $I_1 \sim I_k \sim I_2$ of length two. Thus, $d(I_1, I_2) = 2$. 
		
		\textbf{Subcase 2.} $I_k \subset I_1$. There exists $x \in I_1$ but $x \notin I_k $. If $I_2 \cup I_k = S$ then $x \in I_2$. Thus, we  get $I_1 \cap I_2$ is nontrivial, a contradiction. Consequently, $I_2 \cup I_k \neq S$. Further, we get a path $I_1  \sim (I_2 \cup I_k) \sim I_2$ of length two. Thus, $d(I_1, I_2) = 2$.
		
		\textbf{Subcase 3.} $I_1 \subset I_k$. Since $I_1 \cup I_2 = S$ we get $I_k \cup I_2 = S$. Further, the intersection of  $I_k$ and $I_2$ is nontrivial. Consequently, $I_1 \sim I_k \sim I_2$ gives a path of length two between $I_1$ and $I_2$. Thus, $d(I_1, I_2) = 2$. Hence, $diam(\Gamma(S))$ $\leq$ $2$.
		
	\end{proof}	
\begin{lemma}
  Let $S$ be a semigroup having minimal left ideals. Then $\Gamma(S)$ is complete if and only if $S$ has unique minimal left ideal.
\end{lemma}
\begin{proof}
    Suppose that $S$ contains a unique minimal left ideal $I_1$. Note that every nontrivial left ideal of $S$ contains at least one minimal left ideal. Since $I_1$ is unique then it must contained in every nontrivial left ideals of $S$. Thus, the graph $\Gamma(S)$ is complete. 
        
        Conversely, suppose that $\Gamma(S)$ is a complete graph. On  contrary if $S$ has at least two minimal left ideals, viz. $I_1, I_2$. By Remark \ref{disjoint intersection minimal}, $I_1 \nsim I_2$, a contradiction to the fact that $\Gamma(S)$ is complete. Thus $S$ has unique minimal left ideal.  
\end{proof}

\begin{lemma}
 If graph $\Gamma(S)$ is a regular if and only if either $\Gamma(S)$ is null or a complete graph.   
\end{lemma}
\begin{proof}
        First suppose that $\Gamma(S)$ is not a null graph. Let if possible, $S$ has at least two minimal left ideals, namely $I_1, I_2$. Since $\Gamma(S)$ is not a null graph then $I_1$ and $I_1 \cup I_2$ forms a nontrivial left ideals of $S$ and $I_1 \sim (I_1 \cup I_2)$. Suppose $J$ is any nontrivial left ideal of $S$ such that $J \sim I_1$ then $J \sim (I_1 \cup I_2)$. It follows that every nontrivial left ideal of $S$ which is adjacent with $I_1$ is also adjacent with $(I_1 \cup I_2)$ and $I_2 \sim I_1 \cup I_2$ but $I_2 \nsim I_1$ implies that $deg(I_1) < deg(I_1 \cup I_2)$, a contradiction. Therefore, $\Gamma(S)$ is a complete graph.
\end{proof}
Next we classify the semigroups such that the diameter of intersection ideal graph $\Gamma(S)$ is two.

\begin{theorem}\label{diameter2clasification}
 Let $S$ be a semigroup having minimal left ideals. Then for a connected graph $\Gamma(S)$, we have  $diam(\Gamma(S)) = 2$ if and only if $S$ has at least two minimal left ideals. 
\end{theorem}
\begin{proof}
Suppose that $diam(\Gamma(S)) = 2$. Assume that $I_1$ is the only minimal left ideal of $S$. Since $I_1$ is unique minimal left ideal then it is contained in all other nontrivial left ideals of $S$. Therefore, for any nontrivial left ideals $J, K$, we have $I_1 \subset (J \cap K)$. Consequently, $d(J, K) = 1$ for any $J, K \in V(\Gamma(S))$. Therefore $S$ has at least two minimal left ideals. Conversely suppose that $S$ has at least two minimal left ideals, viz. $I_1, I_2$. Then by Remark \ref{disjoint intersection minimal}, we have $I_1 \nsim I_2$. Consequently, by Theorem \ref{diameter2}, $d(I_1, I_2) = 2$. Thus,  $diam(\Gamma(S)) = 2$. 
\end{proof}

\section{Invariants of  $\Gamma(S)$}
In this section, first we obtain the girth of $\Gamma(S)$. Then we discuss planarity and perfectness of $\Gamma(S)$. Also we classify the semigroup $S$ such that  $\Gamma(S)$ is bipartite, star graph and tree, respectively. Further, we investigate other graph invariants viz. clique number, independence number and  strong metric dimension of $\Gamma(S)$.
\begin{theorem} \label{girthofintersection}
Let $S$ be a semigroup such that $\Gamma(S)$ contains a cycle. Then $g(\Gamma(S)) = 3$.
\end{theorem}
\begin{proof}
If $\Gamma(S)$ is disconnected or a tree, then clearly $g(\Gamma(S)) = \infty$. Suppose that the semigroup $S$ has $n$ minimal left ideals. Now we prove the result through following cases.

\noindent\textbf{Case 1.} $n = 0$. If $S$ has no nontrivial left ideals then there is nothing to prove. Otherwise, there exists a chain of nontrivial left ideals of $S$ such that $I_1 \supset I_2 \supset \cdots \supset I_k \supset \cdots$. Thus, $g(\Gamma(S)) = 3$.

\noindent\textbf{Case 2.} $n = 1$. Suppose that $I_1$ is the only minimal left ideal of $S$. Since $I_1$ is unique minimal left ideal then it is contained in all other nontrivial left ideals of $S$. Therefore, for any nontrivial left ideals $I, J$, we have $I_1 \subset I \cap J \neq \emptyset$.  If $S$ has at least three nontrivial left ideals then $g(\Gamma(S)) = 3$. Otherwise, $g(\Gamma(S)) = \infty$.

\noindent\textbf{Case 3.} $n = 2$. Let $I_1, I_2$ be two minimal left ideals of $S$. If $I_1 \cup I_2 = S$ then by Theorem \ref{two minimalintersection} and Corollary \ref{null graphintersection}, $g(\Gamma(S)) = \infty$. If $I_1 \cup I_2 \neq S$, then $J = I_1 \cup I_2$ is a nontrivial left ideal of $S$. If $S$ has only these three, namely  $I_1, I_2$ and $J$, left ideals then we obtain $I_1 \sim J \sim I_2$ so that $g(\Gamma(S)) = \infty$. Now suppose that  $S$ has a nontrivial left ideal $K$ other than $I_1, I_2$ and $J$. Since  $I_1, I_2$ are minimal left ideals of $S$ we have either $I_1 \subset K$ or $I_2 \subset K$. Without loss of generality, assume that $I_1 \subset K$, then we get a triangle $I_1 \sim K \sim J \sim I_1$. It follows that $g(\Gamma(S)) = 3$.

 \noindent\textbf{Case 4.} $n \geq 3$. Let $I_1, I_2, I_3$ be the minimal left ideals of $S$. Then we have a cycle $(I_1 \cup I_2) \sim (I_2 \cup I_3) \sim (I_1 \cup I_3) \sim (I_1 \cup I_2)$ of length 3. Thus, $g(\Gamma(S)) = 3$. 
\end{proof}

	Let ${\rm Min}(S)$ (${\rm Max}(S)$) be the set of all minimal (maximal) left ideals of $S$. For a nontrivial left ideal $I_{{i_1}{i_2}\cdots{i_k}}$, we mean $I_{i_1} \cup I_{i_2} \cup \cdots \cup I_{i_k}$, where $I_{i_1}, I_{i_2}, \cdots, I_{i_k}$ $\in {\rm Min}(S)$ %are minimal left ideals of $S$.
	\begin{theorem}
		For the graph $\Gamma(S)$, we have the following results:
		
		\begin{enumerate}
		 \item[{\rm(i)}] If $\Gamma(S)$ is  planar  then $| {\rm Min}(S) | \leq 3$.
			\item [{\rm(ii)}] For $S = I_{{i_1}{i_2}\cdots{i_n}}$, we have  $\Gamma(S)$ is planar if and only if $n \leq 3$.  
	\end{enumerate}
	\end{theorem}
	\begin{proof}
		(i) Suppose that $| {\rm Min}(S) | =4$ with ${\rm Min}(S) = \{I_1, I_2, I_3, I_4\}$. Then note that the subgraph induced by the vertices $I_1, I_{12}, I_{123}, I_{14}$ and $I_{124}$ is isomorphic to $K_5$. Thus, $\Gamma(S)$ is nonplanar.
		
		(ii) The proof for $\Gamma(S)$ is nonplanar for $n \geq 4$ follows  from part (i).
		If $n = 2$ then by Corollary \ref{null graphintersection} and Theorem \ref{two minimalintersection}, $\Gamma(S)$ is planar. For $n= 3$, $\Gamma(S)$ is planar as shown in Figure 1.
		\begin{figure}[h!]
			\centering
			\includegraphics[width=0.2 \textwidth]{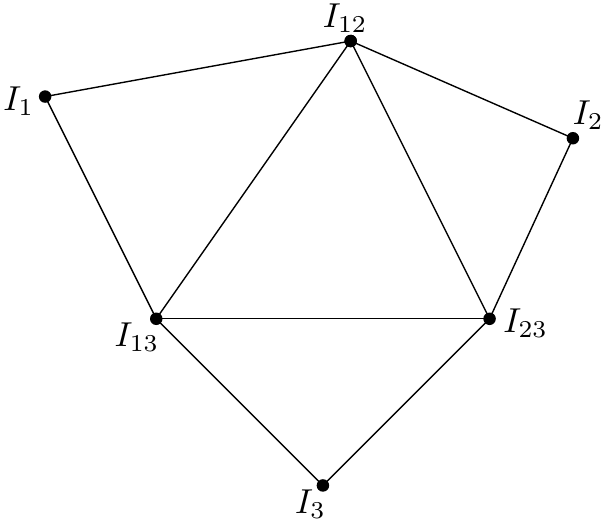}
			\caption{Planar drawing of $\Gamma(S)$ for $S = I_{123}$}
	\end{figure}
		\end{proof}
		
\begin{theorem}
		For the graph $\Gamma(S)$, we have the following results:
		
		\begin{enumerate}
		 \item[{\rm(i)}] If 
		 $\Gamma(S)$ is a perfect graph then $| {\rm Min}(S) | \leq 4$.
			\item [{\rm(ii)}] Let $S$ be the union of $n$ minimal left ideals. Then  $\Gamma(S)$ is perfect if and only if $n \leq 4$.  
	\end{enumerate}
	\end{theorem}
	\begin{proof}
	(i) Suppose that $| {\rm Min}(S) | = 5$ with ${\rm Min}(S) = \{I_1, I_2, I_3, I_4, I_5 \}$. Note that $I_{12} \sim I_{23} \sim I_{34} \sim I_{45} \sim I_{15} \sim I_{12}$ induces a cycle of length 5. Then by Theorem \ref{strongperfecttheorem}, $\Gamma(S)$ is not  perfect.
	
	(ii) The proof for $\Gamma(S)$ is not a perfect graph for $n \geq 5$ follows  from part (i). If $n = 2$ then by Corollary \ref{null graphintersection} and Theorem \ref{disconnectedintersection}, $\Gamma(S)$ is disconnected. Thus, being a null graph, $\Gamma(S)$ is perfect. For $n \in \{ 3, 4\}$, we show that $\Gamma(S)$ does not contain a hole or an antihole of odd length at least five (cf. Theorem \ref{strongperfecttheorem}). If $n =3$, $\Gamma(S)$ is perfect as shown in Figure 1. If $n = 4$ then  from Figure 2 note that  $\Gamma(S)$ does not contain a hole or an antihole of odd length at least five. 
	\begin{figure}[h!]
			\centering
			\includegraphics[width=0.5 \textwidth]{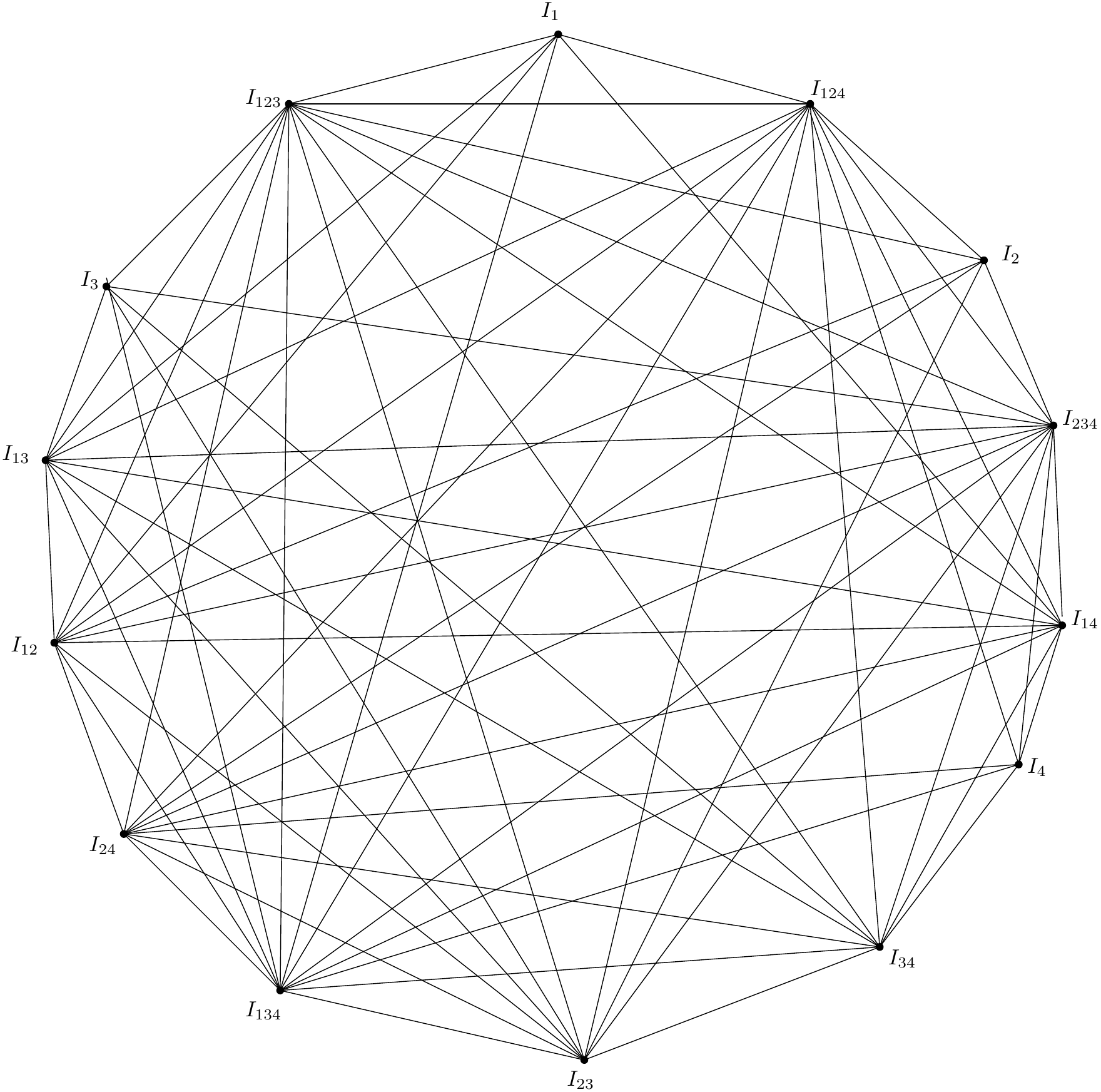}
			\caption{}
	\end{figure}%On contrary, assume that $\Gamma(S)$ contains a hole $C : J_1 \sim J_2 \sim J_3 \sim \cdots \sim J_{2m+1} \sim J_1$, where $m \geq 2$. Since $J_1 \nsim J_3$, $J_1 \nsim J_4$, $\cdots$, $J_1 \nsim J_{2m}$ so that $J_1 \cap J_k = \emptyset$ for $3 \leq k \leq 2m$. Suppose that $J_1$ is a minimal left ideal of $S$. Since $J_1 \sim J_2$ and $J_1 \sim J_{2m+1}$, then by minimality of $J_1$, we have $J_1 \subset J_2$ and $J_1 \subset J_{2m+1}$. It follows that $J_1 \subset J_2 \cap J_{2m+1}$ so that $J_2 \sim J_{2m+1}$, a contradiction. So, $J_1$ is not minimal. Similarly, we can show that $J_k$ are not minimal left ideals of $S$, where $1 \leq k \leq 2m+1$. It follows that $J_k \in \{I_{12}, I_{13}, I_{23} \}$ for $n =3$, where $1 \leq k \leq {2m+1}$ and $m \geq 2$. Clearly, $\Gamma(S)$ does not contain a hole. Now for $n =4$, note that for $1 \leq k \leq 2m+1$, $J_k$'s are either union of two minimal left ideals or three minimal left ideals of $S$. Assume that $J_1 = I_{{i_1}{i_2}{i_3}}$ and $i_s \in \{1, 2, 3, 4\}$. Since $J_1 \nsim J_3$ so that $J_3$ is minimal left ideal of $S$, which is not possible. Similarly, we can show that none of $J_k$'s are union of three minimal left ideals. Therefore, $J_k \in \{ I_{12}, I_{13}, I_{14}, I_{23}, I_{24}, I_{34}\}$. Clearly, $\Gamma(S)$ does not contain a hole.

	\end{proof}
	
	\begin{theorem}
	Let $S$ be a semigroup having minimal left ideals such that $V(\Gamma(S)) > 1$. Then the following conditions are equivalent:
	
	\begin{enumerate}
		 \item[{\rm(i)}] $\Gamma(S)$ is star graph.
		 
		 \item[{\rm(ii)}] $\Gamma(S)$ is a tree.
		 
		 \item[{\rm(iii)}] $\Gamma(S)$ is bipartite.
		 
		  \item[{\rm(iv)}] Either $S$ has exactly three nontrivial left ideals $I_1$, $I_2$ and $I_{12}$ such that $I_1$ and $I_{2}$ are minimal or $S$ has two nontrivial left ideals $I_1, I_2$ such that $I_1 \subset I_2$.
		 \end{enumerate}
	\end{theorem}
	\begin{proof}
	  %(i) $\Rightarrow$  (ii) and (i) $\Rightarrow$ (iii) is straightforward. 
	  We prove (ii), (iii) $\Rightarrow$ (iv). The proof of remaining parts is straightforward. Suppose  $\Gamma(S)$ is a tree. Then clearly $| {\rm Min}(S) | \leq 2$. Otherwise, for minimal left ideals $I_1, I_2, I_3$ we have $I_{12} \sim I_{13} \sim I_{23} \sim I_{12}$ a cycle, a contradiction. Suppose that $| {\rm Min}(S) | = 1$. Let $I_1$ be the unique minimal left ideal of $S$. Consequently, $I_1$ is contained in all other nontrivial left ideals of $S$. If $S$ has at least three nontrivial left ideals then we get a cycle, a contradiction. Thus $|V(\Gamma(S))| = 2$. Now we assume that $| {\rm Min}(S) | = 2$. Let $I_1, I_2$ be two minimal left ideals of $S$. Let if possible, $S = I_{12}$. Then by Corollary \ref{null graphintersection} and Theorem \ref{two minimalintersection}, $\Gamma(S)$ is disconnected so is not a tree. Thus $S \neq I_{12}$.  Then $J = I_{12}$ is a nontrivial left ideal of $S$. %Suppose that $S$ has only these three, namely, $I_1, I_2$ and $J$, left ideals we get $I_1 \sim J \sim I_2$.
	  Now if $S$ has a nontrivial left ideal $K$ other than $I_1, I_2$ and $J$. Without loss of generality, assume that $I_1 \subset K$ then we get a cycle $I_1 \sim I_{12} \sim K \sim I_1$, a contradiction. Thus, for $S \neq I_{12}$, we have $V(\Gamma(S)) = \{I_1, I_2, I_{12}  \}$. %Further  It follows that $S$ has two nontrivial left ideal one is minimal and other is maximal and $| V(\Gamma(S)) | =2$.
	  
	 % We need to show that (iii) $\Rightarrow$ (iv). Suppose that $\Gamma(S)$ is bipartite graph. Then by Theorem \ref{girthofintersection}, $| {\rm Min}(S) | \leq 2$. Assume that $| {\rm Min}(S) | = 2$. Let $I_1, I_2$ be two minimal left ideals of $S$. Let if possible, $S = I_{12}$ then by Corollary \ref{null graphintersection} and Theorem \ref{two minimalintersection}, $\Gamma(S)$ is disconnected, a contradiction. Therefore, $S \neq I_{12}$ and $J = I_{12}$ forms a nontrivial left ideal of $S$. %If $S$ has only these three, namely, $I_1, I_2$ and $J$, left ideals then we obtain $I_1 \sim J \sim I_2$. It follows that $\Gamma(S) \cong K_{1, 2}$, which is a bipartite graph. 
	  %If $S$ has a nontrivial left ideal $K$ of $S$ other than $I_1, I_2$ and $J$. Then by Theorem \ref{girthofintersection}, $\Gamma(S)$ contains a cycle of odd length. Therefore, it is not a bipartite graph. Thus,  $S \neq I_{12}$ and $V(\Gamma(S)) = \{I_1, I_2, I_{12} \}$. Next suppose that $| {\rm Min}(S) | = 1$. Assume that $I_1$ is the only minimal left ideal of $S$. Since $I_1$ is unique minimal left ideal then it is contained in all other nontrivial left ideals of $S$. Then by Theorem \ref{girthofintersection}, $S$ contains a cycle of odd length if $S$  has at least three nontrivial left ideals. It follows that $S$ has two nontrivial left ideals, one is minimal and other one is maximal left ideal of $S$ and $|V(\Gamma(S))| = 2$.  
	   (iii) $\Rightarrow$ (iv). If $\Gamma(S)$ is bipartite then we have $| {\rm Min}(S) | \leq 2$. In the similar lines of the work discussed above, (iv) holds.
	\end{proof}

\begin{theorem}
Let $S$ be a semigroup with $n$ minimal left ideals. Then the following results hold:
\begin{enumerate}
		 \item[{\rm(i)}] If $S \neq I_{12 \cdots n}$ then $\gamma(\Gamma(S)) = 1$.
		 
		\item[{\rm(ii)}] If $S = I_{12 \cdots n}$ then $\gamma(\Gamma(S)) = 2$.
\end{enumerate}
\end{theorem}
\begin{proof}
     (i) Suppose that $S \neq I_{12 \cdots n}$. It follows that $J = I_{12 \cdots n}$ is a nontrivial left ideal of $S$. It is well known that every nontrivial left ideal of $S$ contains at least one minimal left ideal. Consequently, for any nontrivial left ideal $K$ of $S$, we have $J \cap K$ is nontrivial. Thus, $J$ is a dominating vertex. Hence, $\gamma(\Gamma(S)) = 1$.
     
     (ii) Suppose that $S = I_{12 \cdots n}$. Note that  there is no dominating vertex in $\Gamma(S)$ so that $\gamma(\Gamma(S)) \geq 2$. Now we show that $D = \{I_1, I_{23 \cdots n}\}$ is a dominating set. Since $S$ is the union of $n$ minimal left ideals so any nontrivial left ideal of $S$ is union of these minimal left ideals (cf. Remark \ref{everynontrivial left ideal is union}). Let $J \in V(\Gamma(S)) \setminus D$ be any nontrivial left ideal of $S$. Then $J$ is union of $k$ minimal left ideals of $S$, where $1 \leq k \leq n-1$. If $I_1 \subset J$, then we are done. If $I_1 \not\subset J$ then $J$ must be union of $I_2, I_3, \ldots, I_n$. It follows that intersection of $J$ and $I_{23 \cdots n}$ is nontrivial. Consequently, $J \sim I_{23 \cdots n}$. Thus $D$ is a dominating set. This completes the proof.    
\end{proof}
\begin{theorem}
    Let $S$ be a semigroup with $n$ minimal left ideals. Then $\alpha(\Gamma(S)) = n$.
\end{theorem}
\begin{proof}
  Let ${\rm Min}(S) = \{{I_{i_1} : i_1 \in [n]}\}$ be the set of all minimal left ideals of S. Then, by Remark \ref{disjoint intersection minimal}, ${\rm Min}(S)$ is an independent set of $\Gamma(S)$. It follows that $\alpha(\Gamma(S)) \geq n$. Now we prove that for any arbitrary independent set $U$, we have $|U| \leq n$. Assume that $I \in V(\Gamma(S))$ such that $I \in U$. Since every nontrivial left ideal contains at least one minimal left ideal. Without loss of generality, assume that  $I_{{i_1}{i_2}\cdots{i_k}} \subseteq I$ for some $i_1,i_2, \cdots ,i_k 
 \in [n]$. Then note that $|U| \leq n-k+1$. Otherwise, there exist at least two nontrivial left ideals which are adjacent, a contracdiction. Consequently, we have $|U| \leq n$. Thus, $\alpha(\Gamma(S)) = n$.      
\end{proof}
\begin{lemma}
    Let $S$ be a semigroup with $n~ (\geq 3)$ minimal left ideals. Then there exists a clique in $\Gamma(S)$ of size $n$. 
\end{lemma}	
\begin{proof}
Let $I_1, I_2, \ldots, I_n$ be $n$ minimal left ideals. Consider $\mathcal{C} = \{I_{{i_1}{i_2} \cdots {i_{n-1}}} : i_1, i_2, \ldots, i_{n-1} \in [n]\}$. Clearly, $|\mathcal{C}| = n$. Notice that for any $J , K \in \mathcal{C}$, we have  $J \cap K$ is nontrivial so that $J \sim K$. Thus, $\mathcal{C}$ becomes a clique of size $n$.  
\end{proof}

\begin{theorem}
Let $S$ be a semigroup with $n(>1)$ minimal left ideals. Then $\omega(\Gamma(S)) = n$ if and only if one of the following holds:
\begin{enumerate}
		 \item[{\rm(i)}] $S = I_{123}$.
		 \item[{\rm(ii)}] $S$ has only two minimal left ideals $I_1$ and $I_2$ and a unique maximal left ideal $I_{12}$.
		 \end{enumerate}
\end{theorem}
\begin{proof}
First suppose that $\omega(\Gamma(S)) = n$. Assume that $S$ has $n (\geq 4)$ minimal left ideals, namely $I_1, I_2, \ldots, I_n$. Then $\mathcal{C} = \{I_{{i_1}{i_2}\cdots{i_{n-1}}}, I_{{i_1}{i_2}} : i_1, i_2, \ldots, i_n \in [n] \}$ forms a clique of size greater than $n$ of $\Gamma(S)$. It follows that $\omega(\Gamma(S)) > n$. If $n =3$, assume  that $S \neq I_{123}$. Then $\mathcal{C} = \{I_{12}, I_{13}, I_{23}, I_{123}\}$ forms a clique of size four of $\Gamma(S)$. It follows that $S = I_{123}$.  %Consequently, $\omega(\Gamma(S)) > |{\rm Min}(S)|$. Thus,  $S = I_{123}$. %Note that $V(\Gamma(S)) = \{I_1, I_2, I_3, I_{12}, I_{13}, I_{23}\}$. Clearly, $\mathcal{C} = \{I_{12}, I_{13}, I_{23} \}$ forms a clique of size $n$, which is of maximum size. Thus,  $\omega(\Gamma(S)) = |{\rm Min}(S)|$. 
For $n =2$, we have either $S = I_{12}$ or $S \neq I_{12}$. For $S = I_{12}$, by Corollary \ref{null graphintersection} and by Theorem \ref{two minimalintersection}, $\Gamma(S)$ is disconnected. Thus, $\omega(\Gamma(S)) < n$. Thus $S \neq I_{12}$. If $S$ has a nontrivial left ideal $K \notin \{I_1, I_2, I_{12}\}$ then we get a clique of size three. Therefore, $I_{12}$ is a unique maximal left ideal. Converse follows trivially.%If $V(\Gamma(S)) = \{I_1, I_2, I_{12}\}$. Clearly $\omega(\Gamma(S)) = 2$. Thus, $\omega(\Gamma(S)) = |{\rm Min}(S)|$. 
\end{proof}

%\begin{corollary}\label{cliqueno. greaterthan minimalideals}
%Let $S$ be a semigroup with at least three minimal left ideals. Then $\omega(\Gamma(S)) \geq |{\rm Min}(S)|$.
%\end{corollary}

%Let ${\rm Max}(S)$ be the set of all maximal left ideals of $S$. 

\begin{lemma}\label{maximal ideal intersection nonempty}
 If $\Gamma(S)$ is connected then ${\rm Max}(S)$ forms a clique of $\Gamma(S)$.
\end{lemma}
\begin{proof}
We prove the result by showing that if $J, K \in {\rm Max}(S)$ then $J \sim K$.  
Let if possible, $J \nsim K$. The maximality of $J$ and $K$ follows that $J \cup K = S$. By Lemma \ref{S minus K is lclass}, $S \setminus J$ and $S \setminus K$ are $\mathcal{L}-$classes of $S$. It follows that $J$ and $K$ are  only nontrivial left ideals of $S$. Thus, being a null graph $\Gamma(S)$ is disconnected, a contradiction.
\end{proof}

\begin{theorem}
If $K$ is a maximal left ideal of $S$ such that $deg(K)$ is finite, then $\chi(\Gamma(S)) < \infty$.
\end{theorem}
\begin{proof}
 Let $J$ be an arbitrary nontrivial left ideal of $S$ such that $J \nsim K$. Note that $J$ is minimal left ideal of $S$. On contrary, suppose that $J$ is not a minimal left ideal of $S$. Then there exists a nontrivial left ideal $J'$ of $S$ such that $J' \subset J$. Since $K$ is maximal left ideal of $S$. Consequently, $J' \cup K = S$. It follows that  intersection of $J$ and $K$ is nontrivial, a contradiction. By Remark \ref{disjoint intersection minimal}, we can color all the vertices which are not adjacent with $K$ with one color. Since $deg(K)$ is finite, we have $\chi(\Gamma(S)) < \infty$.
\end{proof}
%(iii) Since $S \neq I_{12 \cdots n}$, where $I_1, I_2, \ldots, I_n$ are minimal left ideals of $S$. Then $ J = I_{12 \cdots n}$ forms a nontrivial left ideal of $S$ containing all minimal left ideals of $S$. It is well known that if $S$ has a minimal left ideal then every nontrivial left ideal of $S$ contains at least one minimal left ideal. Then $J$ has a nontrivial intersection with every nontrivial left ideal of $S$. Therefore $J$ is adjacent to every other vertex of $\Gamma(S)$.

%\section{Intersection graph}
%%%%%%%%%%%%%%%%%%%%%%%%%%%%%%%%%%%%%%%%%%%%%%%%%%%%%%%%%%%%%%%%%%%%%%%%%%%%%%%%%%%%%%%%%%%%%%
\begin{lemma}\label{chromaticnumberintersection}
        For $S = I_{{i_1}{i_2}\cdots{i_n}}$, we have $\omega(\Gamma(S)) = \chi(\Gamma(S)) = 2^{n-1}-1$.
\end{lemma}
\begin{proof}
       First note that $S$ has $2^n-2$ nontrivial left ideals and every nontrivial left ideal of $S$ is of the form $I_{{i_1}{i_2}\cdots{i_k}}$ and $1 \leq k \leq n-1$ (cf. Remark \ref{everynontrivial left ideal is union}). If $n$ is odd then consider $\mathcal{C} = \{I_{{j_1}{j_2}\cdots{j_t}} : \lceil \frac{n}{2} \rceil \leq t\leq n-1\}$. Note that $\mathcal{C}$ forms a clique of size $2^{n-1}-1$. We may now suppose that $n$ is even. Consider $\mathcal{C}_1 = \{I_{{j_1}{j_2}\cdots{j_t}} :   
 \frac{n}{2} + 1 \leq t\leq n-1\}$. Notice that $\mathcal{C}_1$ forms a clique. Further, observe that $\mathcal{C}^{'} = \{I_{{i_1}{i_2}\cdots{i_{\frac{n}{2}}}} :  i_1, i_2, \ldots, i_{\frac{n}{2}} \in [n]\}$ do not form a clique because for $j_1, j_2, \ldots, j_{\frac{n}{2}} \in [n] \setminus \{i_1, i_2, \ldots, i_{\frac{n}{2}}\}$, $I_{{i_1}{i_2}\cdots{i_{\frac{n}{2}}}} \nsim I_{{j_1}{j_2}\cdots{j_{\frac{n}{2}}}}$. However, $\mathcal{C}^{''} = \{I_{{i_1}{i_2}\cdots{i_{\frac{n}{2}}}} \in \mathcal{C}^{'} \setminus \{I_{{j_1}{j_2}\cdots{j_{\frac{n}{2}}}}\} : j_1, j_2, \ldots, j_{\frac{n}{2}} \notin \{i_1, i_2, \ldots, i_{\frac{n}{2}}\} \}$ forms a clique of size $\frac{|\mathcal{C}^{'}|}{2}$. Further note that the set $\mathcal{C}_1 \cup \mathcal{C}^{''}$ also forms a clique of size $2^{n-1}-1$. Thus, $\omega(\Gamma(S)) \geq 2^{n-1}-1$. To complete the proof, we show that  $\chi(\Gamma(S)) \leq 2^{n-1}-1$. For $I = I_{{i_1}{i_2}\cdots{i_k}}$ and $J = I_{{j_1}{j_2}\cdots{j_{n-k}}}$, where $i_1, i_2, \ldots, i_k \in [n] \setminus \{j_1, j_2, \ldots, j_{n-k}\}$ we have $I \cap J$ is trivial. Consequently, we can color these vertices with same color so that we can cover all the vertices with $2^{n-1}-1$ colors. Thus $\chi(\Gamma(S)) \leq 2^{n-1}-1$. Hence $\omega(\Gamma(S)) = \chi(\Gamma(S)) = 2^{n-1}-1$.
\end{proof}
%%%%%%%%%%%%%%%%%%%%%%%%%%%%%%%%%%%%%%%%%%%%%%%%%%%%%%%%%%%%%%%
\begin{corollary}
	If $S = I_{{i_1}{i_2}\cdots{i_n}}$  then $\Gamma(S)$ is a weakly perfect graph.
	\end{corollary}
	%%%%%%%%%%%%%%%%%%%%%%%%%%%%%%%%%%%%%%%%%%%%%%%%%%%%%%%%%%%
In order to find the upper bound of the chromatic number of $\Gamma(S)$, where $S$ is an arbitrary semigroup, first we define 
\begin{align*} 
    X_1 & = \{I \in V(\Gamma(S)) :  I_{{i_1}{i_2}\cdots{i_n}} \subseteq I \}\\
        X_2 & = \{I \in V(\Gamma(S)) :  I \subset I_{{i_1}{i_2}\cdots{i_n}} ~\text{and}~ I \neq I_{{i_1}{i_2}\cdots{i_n}} \}\\
X_3 & = V(\Gamma(S)) \setminus (X_1 \cup X_2). 
\end{align*}

Let $\widetilde{{\text{Min}(I)}}$ be the set of all minimal left ideals contained in $I$
. Further define a relation $\rho$ on $X_3$ such that \begin{center}
    $J ~~ \rho ~~K \Longleftrightarrow \widetilde{{\text{Min}(J)}} = \widetilde{{\text{Min}(K)}}$ 
\end{center}
Note that $\rho$ is an equivalence relation.
\begin{theorem}
    Let $S$ be a semigroup with $n$ minimal left ideals and $\chi(\Gamma(S)) < \infty$. Then
    \begin{center}$\chi(\Gamma(S)) \leq |X_1| + (2^{n-1}-1) + (2^{n-1}-1)m$,
    \end{center}
    where $m = {\rm{max}} \{|C(I)| : C(I) \
    \text{is an equivalence class of} \; \rho \}$.
\end{theorem}
\begin{proof}
       %Consider the sets \begin{center}
      %  $A = \{I \in V(\Gamma(S)) :  I_{{i_1}{i_2}\cdots{i_n}} \subseteq I \}$\\
       % $B = \{I \in V(\Gamma(S)) :  I \subset I_{{i_1}{i_2}\cdots{i_n}} ~\text{and}~ I \neq I_{{i_1}{i_2}\cdots{i_n}} \}$\\
        
       % $C = V(\Gamma(S)) \setminus (A \cup B)$.
      % \end{center}
       Note that for any $I, J \in X_1$, we have $I \sim J$. Since every nontrivial left ideal contains at least one minimal left ideal, consequently each element of $X_1$ is a dominating vertex of $\Gamma(S)$. Therefore, we need at least $|X_1|$ colors in any coloring of $\Gamma(S)$. %Suppose that $|| = m_1$. 
       By proof of Lemma \ref{chromaticnumberintersection}, we can color all the vertices of $X_2$ with at least $2^{n-1}-1$ colors so that we need at least $2^{n-1}-1 + |X_1|$ colors to color $X_1 \cup X_2$. 
       
     To prove our result we need to show that the vertices of $X_3$ can be colored by using $(2^{n-1}-1)m$ colors. Now let $J, K \in X_3$ such that $I_{{i_1}{i_2}\cdots{i_k}} \subset J$ and $I_{{j_1}{j_2}\cdots{j_t}} \subset K$. Note that $J \cap K$ is nontrivial if and only if $I_{{i_1}{i_2}\cdots{i_k}} \cap I_{{j_1}{j_2}\cdots{j_t}}$ is nontrivial. It follows that $J \sim K$ in $\Gamma(S)$ if and only if either $I_{{i_1}{i_2}\cdots{i_k}} = I_{{j_1}{j_2}\cdots{j_t}}$ or $I_{{i_1}{i_2}\cdots{i_k}} \sim I_{{j_1}{j_2}\cdots{j_t}}$.
       
      %$\mathcal{R}$ on $C$ as follows: $I_1 \mathcal{R} I_2$ if and only if $I_{{i_1}{i_2}\cdots{i_k}} = I_{{j_1}{j_2}\cdots{j_t}}$. Observe that $\mathcal{R}$ is an equivalence relation whose equivalence classes are
      Note that the equivalence class of $I \in X_3$ under $\rho$ is $C(I) = \{J \in X_3 : \widetilde{{\text{Min}(I)}} = \widetilde{{\text{Min}(J)}} \}$. Since $\chi(\Gamma(S)) < \infty$ we get $|C(I)| < \infty$. Consequently, $|C(I)| \leq m$. Observe that $C(I)$ forms a clique, we require maximum $m$ colors to color each class under $\rho$. Note that $J \in C(J)$ and $K \in C(K)$ such that $J \sim K$ if and only if $I_{{i_1}{i_2}\cdots{i_k}} \sim I_{{j_1}{j_2}\cdots{j_t}}$ in $\Gamma(S)$. Consequently, we can color the vertices in $X_3$ by using $(2^{n-1}-1)m$ colors. 
       \end{proof}
       
       \begin{theorem}
    Let $S$ be a semigroup with $n$ minimal left ideals. Then 
 \[\operatorname{sdim}(\Gamma(S)) =  \begin{cases}
|X_1| + |X_3| + 2^{n-1}-2 &  \text{\rm if} ~ S \neq I_{{i_1}{i_2}\cdots{i_n}} \\
2^{n-1}-1 &  \text{\rm if}~ S = I_{{i_1}{i_2}\cdots{i_n}}
\end{cases}\]
\end{theorem}
\begin{proof}
        Let $I, J \in V(\Gamma(S))$ such that $I_{{i_1}{i_2}\cdots{i_k}} \subseteq I$ and $I_{{j_1}{j_2}\cdots{j_t}} \subseteq J$. Then $I \sim J$ if and only if either $I_{{i_1}{i_2}\cdots{i_k}} = I_{{j_1}{j_2}\cdots{j_t}}$ or $I_{{i_1}{i_2}\cdots{i_k}} \sim I_{{j_1}{j_2}\cdots{j_t}}$. Define a relation $\rho_1$ on $V(\Gamma(S))$ such that $I$  $\rho_1$  $J$ if and only if $\widetilde{{\text{Min}(I)}} = \widetilde{{\text{Min}(J)}}$. Clearly, $\rho_1$ is an equivalence relation on $V(\Gamma(S))$. Let $N[I_{{i_1}{i_2}\cdots{i_k}}] = \{I \in V(\Gamma(S)) : \widetilde{{\text{Min}(I)}} = I_{{i_1}{i_2}\cdots{i_k}}\}$  be equivalence class containing $I_{{i_1}{i_2}\cdots{i_k}}$. If $S \neq  I_{{i_1}{i_2}\cdots{i_n}}$, then by Theorem \ref{strong-metric-dim}, we have $\mathcal{R}_{\Gamma(S)}$ whose vertex set $V(\mathcal{R}_{\Gamma(S)}) = \{I_{{i_1}{i_2}\cdots{i_k}} : i_1, i_2, \cdots, i_k \in [n] ~  \text{and} ~ 1 \leq k \leq n\}$. By using Lemma \ref{chromaticnumberintersection}, note that $\omega (\mathcal{R}_{\Gamma(S)}) = 2^{n-1}$. Then $\operatorname{sdim}(\Gamma(S)) = |X_1| + |X_3| + 2^{n-1}-2$. Next, if $S = I_{{i_1}{i_2}\cdots{i_n}}$, then $V(\mathcal{R}_{\Gamma(S)}) = \{I_{{i_1}{i_2}\cdots{i_k}} : i_1, i_2, \cdots, i_k \in [n] ~  \text{and} ~ 1 \leq k \leq n-1\}$. By using Lemma \ref{chromaticnumberintersection}, note that $\omega (\mathcal{R}_{\Gamma(S)}) = 2^{n-1}-1$. Then $\operatorname{sdim}(\Gamma(S)) =  2^{n-1}-1$.
\end{proof}
%%%%%%%%%%%%%%%%%%%%%%%%%%%%%%%%%%%%%%%%%%%%%%%%%%%%%%%%%%%%%%%%%%%%%%%%%%%%%%%%%%%%%%%%%%%%%%
Now in the remaining section, we consider a class of those semigroups which are union of $n$ minimal left ideals. In particular, completely simple semigroups belongs to this class. In what follows, the semigroup $S$ is assumed to be the union of $n$ minimal left ideals $I_{i_1}, I_{i_2}, \ldots, I_{i_n}$ i.e. $S = I_{{i_1}{i_2}\cdots{i_n}}$. The following lemma gives the lower bound of the metric dimension of $\Gamma(S)$. 
\begin{lemma}[{\cite[Theorem 1]{chartrand2000resolvability}}]\label{metric dimension theorem}
        For positive integers $d$ and $m$ with $d < m$, define $f(m, d)$ as the least positive integer $k$ such that $k + d^k \geq m$. Then for a connected graph $\Gamma$ of order $m \geq 2$ and diameter $d$, the metric dimension $\beta(\Gamma) \geq f(m, d)$.   
\end{lemma}

\begin{theorem}
    If $S = I_{{i_1}{i_2}\cdots{i_n}}$ then  the metric dimension of $\Gamma(S)$ is given below:
\[\beta(\Gamma(S)) = \begin{cases}
2 &  \text{\rm if} ~ n = 3 \\
n &  \text{\rm if}~ n \geq 4
\end{cases}\]
\end{theorem}
\begin{proof} 
For $n =3$, it is easy to observe that $\{I_{i_1}, I_{i_2}\}$ forms a minimum resolving set. If $n \geq 4$ then by Remark \ref{everynontrivial left ideal is union}, we have $|V(\Gamma(S))| = 2^n-2$. In view of Lemma \ref{metric dimension theorem}, we get \begin{center}
 $n = f(2^n - 2, 2) \leq  \beta(\Gamma(S))$.  
\end{center}
It is easy to observe that for $k = n-1$, $2^k + k \not\geq 2^n - 2$. Therefore, the least positive integer $k$ is $n$ for which $k + 2^k \geq 2^n-2$. Thus $n \leq \beta(\Gamma(S))$.
To obtain upper bound of $\beta(\Gamma(S))$, let $J = I_{{i_1}{i_2}\cdots{i_k}}$ and $K = I_{{j_1}{j_2}\cdots{j_t}}$ be distinct arbitrary vertices $\Gamma(S)$. Since $J \neq K$, there exists at least $I_{i_s} \in {\rm Min}(S)$ such that $I_{i_s} \sim J$ and $I_{i_s} \nsim K$. It follows that $d(J, I_{i_s}) \neq d(K, I_{i_s})$. Thus ${\rm Min}(S) = \{I_{i_1} : i_1 \in [n]\}$ forms a resolving set for $\Gamma(S)$ of size $n$. It follows that $\beta(\Gamma(S)) \leq n$. This completes our proof.
\end{proof}

An automorphism of a graph $\Gamma$  is a permutation $f$ on $V (\Gamma)$ with the property that, for any vertices $u$ and $v$, we
	have $uf \sim vf$ if and only if $u \sim v$. The set $Aut(\Gamma)$ of all graph automorphisms of a graph $\Gamma$ forms a group with
	respect to composition of mappings. The symmetric group of degree $n$ is denoted by $S_n$. Now we obtain the automorphism group of $\Gamma(S)$, when $S$ is union of $n$ minimal left ideal. 
	\begin{lemma}\label{degree k}
	        Let $S = I_{{i_1}{i_2} \cdots {i_{n}}}$ and let $K = I_{{i_1}{i_2} \cdots {i_{k}}}$ be a nontrivial left ideal of $S$. Then $deg(K) = (2^k-2) + (2^{n-k}-2) + (2^{n-k}-1)(2^{k}-2)$.
	\end{lemma}
	\begin{proof}
	   Let $J$ be a nontrivial left ideal of $S$  such that $J \sim K$. Clearly $J \cap K$ is a nontrivial left ideal. Now we discuss the following cases:
		
		\noindent\textbf{Case 1.}  $J \not\subset K$ and $K \not\subset J$. Since $J \sim K$ and $K = I_{{i_1}{i_2} \cdots {i_{k}}}$ then note that the number of nontrivial left ideals such that $J \not\subset K$ and $K \not\subset J$ is 
	 %\begin{center} 
	   %\hspace{4.2cm}= $^k C_1(^{n-k} C_1 + ^{n-k} C_2 + \cdots + ^{n-k} C_{n-k}) + ^k C_2(^{n-k} C_1 + ^{n-k} C_2 + \cdots + ^{n-k} C_{n-k}) + \cdots + ^k C_{k-1}(^{n-k} C_1 + ^{n-k} C_2 + \cdots + ^{n-k} C_{n-k})$.\\
\begin{align*}
 &= \left(\sum_{i=1}^{n-k} \binom{n-k}{i}\right) \left(\sum_{i=1}^{k-1} \binom{k}{i}\right) = (2^{n-k}-1)(2^k-2) 
 \end{align*}

	 %\end{center}
	 %Consequently, we have $(\sum_{i=1}^{n-k} \binom{n-k}{i}) (\sum_{i=1}^{k-1} \binom{k}{i}) = (2^{n-k}-1)(2^k-2)$ nontrivial left ideals which have nontrivial intersection.
	 
\noindent\textbf{Case 2.} $J \subset K$. The number of nontrivial left ideals of $S$ which are properly contained in $K$ are $2^k-2$.

\noindent\textbf{Case 3.} $K \subset J$. The number of nontrivial left ideals of $S$ properly containing $K$ are $2^{n-k}-2$.
Thus, from  the above cases we have the result.    
	\end{proof}
	
	\begin{corollary}
	If $S = I_{{i_1}{i_2}\cdots{i_n}}$  then the graph $\Gamma(S)$ is Eulerian for $n \geq 3$.
	\end{corollary}
	
\begin{lemma}\label{symmetric group}
		For $\sigma \in S_n$, let $\phi_{\sigma} : V(\Gamma(S)) \rightarrow V(\Gamma(S))$ defined by $\phi_{\sigma}(I_{{i_1}{i_2}\cdots {i_k}}) = I_{\sigma({i_1})\sigma({i_2})\cdots \sigma({i_k})}$. Then $\phi_{\sigma} \in Aut(\Gamma(S))$.
	\end{lemma}
	\begin{proof}
	It is easy to verify that $\phi_{\sigma}$ is a permutation on $V(\Gamma(S))$. Now we show that $\phi_{\sigma}$ preserves adjacency. Let $I_{{i_1}{i_2}\cdots {i_t}}$ and $I_{{j_1}{j_2}\cdots {j_k}}$ be arbitrary vertices of $\Gamma(S)$ such that $I_{{i_1}{i_2}\cdots {i_t}} \sim I_{{j_1}{j_2}\cdots {j_k}}$. Then  $I_{{i_1}{i_2}\cdots {i_t}} \cap I_{{j_1}{j_2}\cdots {j_k}} \neq \emptyset$.
Now 
\begin{align*}
I_{{i_1}{i_2}\cdots {i_t}} \sim I_{{j_1}{j_2}\cdots {j_k}} 
&\Longleftrightarrow I_{\sigma({i_1})\sigma({i_2})\cdots \sigma({i_t})} \sim I_{\sigma({j_1})\sigma({j_2})\cdots \sigma({j_k})}\\
& \Longleftrightarrow \phi_{\sigma}(I_{{i_1}{i_2}\cdots {i_t}}) \sim \phi_{\sigma}(I_{{j_1}{j_2}\cdots {j_k}}).
\end{align*}
Thus, $\phi_{\sigma} \in Aut(\Gamma(S))$.
\end{proof}
	
	\begin{proposition}\label{phisigma}
		For each $f \in  Aut(\Gamma(S))$, we have  $f = \phi_{\sigma}$ for some $\sigma \in S_n$.
	\end{proposition}
	\begin{proof}
	 In view of Lemma \ref{degree k} and Lemma \ref{symmetric group}, suppose that
	 $f(I_{i_1}) = I_{j_1}$, $f(I_{i_2}) = I_{j_2}$, $\ldots$, $f(I_{i_n}) = I_{j_n}$. Consider $\sigma \in S_n$ such that $\sigma(i_1) = j_1, \sigma(i_2) = j_2, \ldots, \sigma(i_n) = j_n$. Then $\phi_{\sigma}(I_{{i_1}{i_2}\cdots{i_k}}) = I_{\sigma({i_1})\sigma({i_2})\cdots \sigma({i_k})} = I_{{j_1}{j_2}\cdots{j_k}}$ (cf. Lemma \ref{symmetric group}). Clearly, $I_{i_1} \sim I_{{i_1}{i_2}\cdots{i_k}}$, $I_{i_2} \sim I_{{i_1}{i_2}\cdots{i_k}}$, $\ldots$, $I_{i_k} \sim I_{{i_1}{i_2}\cdots{i_k}}$. Also note that    $I_{i_t} \cap I_{{i_1}{i_2}\cdots{i_k}}$ is trivial for ${i_t} \in \{i_{k+1}, i_{k+2}, \ldots, i_{n}\}$  where $i_{k+1}, i_{k+2}, \ldots, i_{n}\in [n] \setminus \{i_1, i_2, \ldots, i_k\}$. It follows that $I_{i_{k+1}} \nsim I_{{i_1}{i_2}\cdots{i_k}}$, $I_{i_{k+2}} \nsim I_{{i_1}{i_2}\cdots{i_k}}$, $\ldots$, $I_{i_{n}} \nsim I_{{i_1}{i_2}\cdots{i_k}}$. Thus, $f(I_{i_1}) \sim f(I_{{i_1}{i_2}\cdots{i_k}})$, $f(I_{i_2}) \sim f(I_{{i_1}{i_2}\cdots{i_k}})$, $\ldots$, $f(I_{i_k}) \sim f(I_{{i_1}{i_2}\cdots{i_k}})$ and $f(I_{i_{k+1}}) \nsim f(I_{{i_1}{i_2}\cdots{i_k}})$, $f(I_{i_{k+2}}) \nsim f(I_{{i_1}{i_2}\cdots{i_k}})$, $\ldots$, $f(I_{i_{n}}) \nsim f(I_{{i_1}{i_2}\cdots{i_k}})$. Consequently, $I_{j_1} \subset f(I_{{i_1}{i_2}\cdots{i_k}})$, $I_{j_2} \subset f(I_{{i_1}{i_2}\cdots{i_k}})$, $\ldots$, $I_{j_k} \subset f(I_{{i_1}{i_2}\cdots{i_k}})$ and $I_{j_{k+1}} \not \subset f(I_{{i_1}{i_2}\cdots{i_k}})$, $I_{j_{k+2}} \not \subset f(I_{{i_1}{i_2}\cdots{i_k}})$, $\ldots$, $I_{j_n} \not \subset f(I_{{i_1}{i_2}\cdots{i_k}})$. It follows that $f(I_{{i_1}{i_2}\cdots{i_k}}) = I_{{j_1}{j_2}\cdots{j_k}} = \phi_{\sigma}(I_{{i_1}{i_2}\cdots{i_k}})$. Thus, $f = \phi_{\sigma}$.
	\end{proof}
	
	\begin{theorem}\label{automorphism group}
		Let $S$ be the union of $n$ minimal left ideals. Then for $n \geq 2$, we have $ Aut(\Gamma(S)) \cong S_n$. Moreover, $|Aut(\Gamma(S))| = n!$.
	\end{theorem}
	
	\begin{proof}
			In view of Lemma \ref{symmetric group} and by Proposition \ref{phisigma}, 
		note that the underlying set of the automorphism group of $\Gamma(S)$ is
		$Aut(\Gamma(S)) = \{\phi_{\sigma} \; : \; \sigma \in S_n \}$, where $S_n$ is a symmetric group of degree $n$. The groups $Aut(\Gamma(S))$ and $S_n$ are isomorphic under the assignment $\phi_{\sigma} \mapsto \sigma$.  Since all the elements in $Aut(\Gamma(S))$ are distinct, we have $|Aut(\Gamma(S))| = n!$.
	\end{proof}
	
\section{Acknowledgement}
The first author gratefully acknowledge for providing financial support to CSIR  (09/719(0093)/2019-EMR-I) government of India. The second author wishes to acknowledge the support of MATRICS Grant  (MTR/2018/000779) funded by SERB, India.


\begin{thebibliography}{10}

\bibitem{ahmadi2016planarity}
H.~Ahmadi and B.~Taeri.
\newblock Planarity of the intersection graph of subgroups of a finite group.
\newblock {\em Journal of Algebra and Its Applications}, 15(03):1650040, 2016.

\bibitem{akbari2015intersectiongroup}
S.~Akbari, F.~Heydari, and M.~Maghasedi.
\newblock The intersection graph of a group.
\newblock {\em Journal of Algebra and Its Applications}, 14(05):1550065, 2015.

\bibitem{akbari2014some}
S.~Akbari and R.~Nikandish.
\newblock Some results on the intersection graph of ideals of matrix algebras.
\newblock {\em Linear and Multilinear Algebra}, 62(2):195--206, 2014.

\bibitem{S.Akbari2013}
S.~Akbari, R.~Nikandish, and M.~J. Nikmehr.
\newblock Some results on the intersection graphs of ideals of rings.
\newblock {\em Journal of Algebra and Its Applications}, 12(4):1250200--13,
  2013.

\bibitem{akbari2012intersection}
S.~Akbari, H.~Tavallaee, and S.~K. Ghezelahmad.
\newblock Intersection graph of submodules of a module.
\newblock {\em Journal of Algebra and Its Applications}, 11(01):1250019, 2012.

\bibitem{akbari2017some}
S.~Akbari, H.~Tavallaee, and S.~K. Ghezelahmad.
\newblock Some results on the intersection graph of submodules of a module.
\newblock {\em Mathematica Slovaca}, 67(2):297, 2017.

\bibitem{i.Bosak}
J.~Bos\'{a}k.
\newblock The graphs of semigroups.
\newblock In {\em Theory of {G}raphs and its {A}pplications ({P}roc. {S}ympos.
  {S}molenice, 1963)}. Publ. House Czechoslovak Acad. Sci., Prague, 1964.

\bibitem{a.sen2009}
I.~Chakrabarty, S.~Ghosh, T.~K. Mukherjee, and M.~K. Sen.
\newblock Intersection graphs of ideals of rings.
\newblock {\em Discrete Math.}, 309:5381--5392, 2009.

\bibitem{chartrand2000resolvability}
G.~Chartrand, L.~Eroh, M.~A. Johnson, and O.~R. Oellermann.
\newblock Resolvability in graphs and the metric dimension of a graph.
\newblock {\em Discrete Applied Mathematics}, 105(1-3):99--113, 2000.

\bibitem{strongperfectgraph}
M.~Chudnovsky, N.~Robertson, P.~Seymour, and R.~Thomas.
\newblock The strong perfect graph theorem.
\newblock {\em Ann. of Math. (2)}, 164:51--229, 2006.

\bibitem{b.clifford61vol1}
A.~H. Clifford and G.~B. Preston.
\newblock {\em The algebraic theory of semigroups. {V}ol. {I}}.
\newblock Mathematical Surveys, No. 7. American Mathematical Society, 1961.

\bibitem{a.Cskany1969}
B.~Cs\'{a}k\'{a}ny and G.~Poll\'{a}k.
\newblock The graph of subgroups of a finite group.
\newblock {\em Czechoslovak Math. J.}, 19 (94):241--247, 1969.

\bibitem{das2017}
A.~Das.
\newblock On perfectness of intersection graphs of ideals of $\mathbb{Z}_n$.
\newblock {\em Discuss. Math. Gen. Algebra Appl.}, 37:119--126, 2017.

\bibitem{a.Haghi2017}
E.~Haghi and A.~R. Ashrafi.
\newblock Note on the cyclic subgroup intersection graph of a finite group.
\newblock {\em Quasigroups Related Systems}, 25(2):245--250, 2017.

\bibitem{MR2660547}
S.~H. Jafari and N.~Jafari~Rad.
\newblock Planarity of intersection graphs of ideals of rings.
\newblock {\em Int. Electron. J. Algebra}, 8:161--166, 2010.

\bibitem{a.jafari}
S.~H. Jafari and N.~Jafari~Rad.
\newblock On the intersection graphs of normal subgroups on nilpotent groups.
\newblock {\em An. Univ. Oradea Fasc. Mat.}, 20(1):17--20, 2013.

\bibitem{jafari2011dominion}
S.~H. Jafari and N.~J. Rad.
\newblock Domination in the intersection graphs of rings and modules.
\newblock {\em Ital. J. Pure Appl. Math}, 28:19--22, 2011.

\bibitem{jafari2011results}
N.~Jafari~Rad and S.~Jafari.
\newblock Results on the intersection graphs of subspaces of a vector space.
\newblock {\em arXiv: 1105.0803v1, http://arxiv.org/abs/1105.0803v1}, 2011.

\bibitem{MR3310566}
N.~Jafari~Rad, S.~H. Jafari, and S.~Ghosh.
\newblock On the intersection graphs of ideals of direct product of rings.
\newblock {\em Discuss. Math. Gen. Algebra Appl.}, 34(2):191--201, 2014.

\bibitem{kayacan2018connectivity}
S.~Kayacan.
\newblock Connectivity of intersection graphs of finite groups.
\newblock {\em Comm. Algebra}, 46(4):1492--1505, 2018.

\bibitem{a.kalyacan}
S.~Kayacan.
\newblock Dominating sets in intersection graphs of finite groups.
\newblock {\em Rocky Mountain J. Math.}, 48(7):2311--2335, 2018.

\bibitem{kayacan2015abelian}
S.~Kayacan and E.~Yaraneri.
\newblock Abelian groups with isomorphic intersection graphs.
\newblock {\em Acta Mathematica Hungarica}, 146(1):107--127, 2015.

\bibitem{laison2010subspace}
J.~D. Laison and Y.~Qing.
\newblock Subspace intersection graphs.
\newblock {\em Discrete Math.}, 310(23):3413--3416, 2010.

\bibitem{lin1969}
Y.-F. Lin.
\newblock A problem of bos{\'a}k concerning the graphs of semigroups.
\newblock {\em Proceedings of the American Mathematical Society}, 1969.

\bibitem{X.Ma}
X.~Ma.
\newblock On the diameter of the intersection graph of a finite simple group.
\newblock {\em Czechoslovak Math. J.}, 66(141)(2):365--370, 2016.

\bibitem{ma2018strong}
X.~Ma, M.~Feng, and K.~Wang.
\newblock The strong metric dimension of the power graph of a finite group.
\newblock {\em Discrete Applied Mathematics}, 239:159--164, 2018.

\bibitem{abc}
B.~Pond\v{e}li\v{c}ek.
\newblock Diameter of a graph of a semigroup.
\newblock {\em \v{C}asposis P\v{e}ch}, 92:206--211, 1967.

\bibitem{MR3190084}
Z.~S. Pucanovi\'{c}, M.~Radovanovi\'{c}, and A.~L. Eri\'{c}.
\newblock On the genus of the intersection graph of ideals of a commutative
  ring.
\newblock {\em Journal of Algebra and Its Applications}, 13(5):1350155--20,
  2014.

\bibitem{a.Shahsavari2017}
H.~Shahsavari and B.~Khosravi.
\newblock On the intersection graph of a finite group.
\newblock {\em Czechoslovak Math. J.}, 67(142)(4):1145--1153, 2017.

\bibitem{R.shen}
R.~Shen.
\newblock Intersection graphs of subgroups of finite groups.
\newblock {\em Czechoslovak Math. J.}, 60(135)(4):945--950, 2010.

\bibitem{T.Chelvam2012}
T.~Tamizh~Chelvam and M.~Sattanathan.
\newblock Subgroup intersection graph of finite abelian groups.
\newblock {\em Trans. Comb.}, 1(3):5--10, 2012.

\bibitem{westgraph}
D.~B. West.
\newblock {\em Introduction to Graph Theory, 2nd edn.}
\newblock (Prentice Hall), 1996.

\bibitem{xu2020automorphism}
F.~Xu, D.~Wong, and F.~Tian.
\newblock Automorphism group of the intersection graph of ideals over a matrix
  ring.
\newblock {\em Linear and Multilinear Algebra,
  https://doi.org/10.1080/03081087.2020.1723473}, 2020.

\bibitem{yaraneri2013intersection}
E.~Yaraneri.
\newblock Intersection graph of a module.
\newblock {\em Journal of Algebra and Its Applications}, 12(05):1250218, 2013.

\bibitem{B.kaka}
B.~Zelinka.
\newblock Intersection graphs of finite abelian groups.
\newblock {\em Czechoslovak Math. J.}, 25(100):171--174, 1975.

\end{thebibliography}
\end{document}